\documentclass{amsart}
\usepackage{amssymb,amsmath,amscd,amsthm}
\usepackage{amsfonts,epsfig,latexsym,graphicx,amssymb}
\usepackage{color}
\usepackage{enumerate}
\usepackage{todonotes}

\usepackage{hyperref}
\usepackage{float}

\date{\today}

\newcommand{\Z}{{\mathbb Z}}
\newcommand{\R}{{\mathbb R}}

\newcommand{\N}{{\mathbb N}}

\renewcommand{\P}{{\mathbb P}}

\newcommand{\E}{{\mathbb E}}

\newcommand{\Prob}{\mathop{\mathbb{P}}\nolimits}

\newcommand{\mM}{\mathcal{M}}

\newcommand{\supp}{\mathop{\mathrm{supp}}\nolimits}

\newcommand{\Leb}{\mathop{\mathrm{Leb}}\nolimits}
\newcommand{\Ind}{1\!\!\mathrm{I}}

\newcommand{\tf}{\widetilde{f}}
\newcommand{\tF}{\widetilde{F}}
\newcommand{\tg}{\widetilde{g}}
\newcommand{\ty}{\widetilde{y}}

\newcommand{\rh}{\rho}
\newcommand{\shift}{\theta}
\newcommand{\Diff}{\mathrm{Diff}}

\newtheorem{theorem}{Theorem}[section]
\newtheorem*{theorem*}{Theorem}
\newtheorem{lemma}[theorem]{Lemma}
\newtheorem{prop}[theorem]{Proposition}
\newtheorem{coro}[theorem]{Corollary}

\newtheorem{example}[theorem]{Example}
\theoremstyle{definition}
\newtheorem{remark}[theorem]{Remark}

\newcommand{\mgr}{\mu}  
\newcommand{\msp}{\nu}  

\newcommand{\Sc}{\mathbb{S}^1}  

\title{The fibered rotation number}

\author[P.\ Duarte]{Pedro Duarte}

\address{Departamento de Ciencias Matematicas,
Faculdade de Ciencias da Universidade de Lisboa, P-1749-016 Lisboa, Portugal}

\thanks{P.\ D.\  was partially supported by FCT - Fundacao  para a Ciencia e a Tecnologia, through the project UIDB/04561/2020.}

\email{pmduarte@fc.ul.pt}

\author[A.\ Gorodetski]{Anton Gorodetski}

\address{Department of Mathematics, University of California, Irvine, CA~92697, USA}

\email{asgor@uci.edu}

\thanks{A.\ G.\ was supported in part by NSF grant DMS--2247966.}

\author[V. Kleptsyn]{Victor Kleptsyn}

\address{CNRS, IRMAR (UMR 6625 du CNRS), University of Rennes}

\email{victor.kleptsyn@univ-rennes1.fr}

\thanks{V.K. was supported in part by ANR Gromeov (ANR-19-CE40-0007) and by Centre Henri Lebesgue (ANR-11-LABX-0020-01)}

\begin{document}

\begin{abstract}
We provide an explicit formula for an increment of the fibered rotation number of a one-parameter family of circle cocycles over any ergodic transformation in terms of invariant measures. As an application, for a family of random dynamical systems on the circle, this gives a formula for an increment of the rotation number in terms of the stationary measures. In the case of projective Schr\"odinger cocycles associated with the Anderson Model, that provides a relation between the properties of the stationary measures on the projective space and the integrated density of states (IDS) of the corresponding family of operators. In particular, it gives a dynamical proof of H\"older regularity of the IDS in Anderson Model. 
Finally, we prove that the IDS for the Anderson Model with an ergodic background must be H\"older continuous.
\end{abstract}

\maketitle

\section{Introduction}\label{s.intro}

Since 1885, when H.\,Poincare introduced the notion of the rotation number of a homeomorphism $f:\Sc\to \Sc$ of the circle \cite{Po}, it turned into one of the most basic tools in dynamical systems. Usually it is defined via time averaged shift by a homeomorphism along its orbits, i.e. $\rho(f)=\lim_{n\to \infty}\frac{\tf^n(\tilde x)-\tilde x}{n}$, where $\tf:\R\to \R$ is a lift of the homeomorphism  $f:\Sc\to \Sc$. Due to Birkhoff Ergodic Theorem, it can also be defined via space average of such shifts with respect to an $f$-invariant measure $\nu$:
$$
\rho(f)=\int_{\Sc}(\tf(\tilde x)-\tilde x)d\nu(x).
$$
The main result of this paper is an explicit formula for an increment of the  fibered rotation number in terms of invariant measures of the corresponding cocycle, presented in Theorem \ref{t.main} below. The fibered rotation number, see Section \ref{s.main} for the formal definition, appears naturally in many models, and was studied or used in a variety of specific cases \cite{AFK, AK, BJ, DMW, Gou, H, HA}. In particular, it plays an important role in the spectral theory of ergodic Schr\"odiger operators, where it is closely related to the integrated density of states (IDS), see \cite{ABD, ADG,  CS1, CS2, JM, LW}.

We provide also two applications of this formula. The first one is an explicit expression for an increment of the rotation number of a one parameter family of random dynamical systems on the circle, Theorem \ref{t.main2}, and, as a corollary, the statement on H\"older continuity of the rotation number as a function of the parameter in this case, Theorem \ref{c:iid-Holder}. The latter statement should be considered as a non-linear analog of the well known  in spectral theory result on H\"older regularity of the IDS in Anderson Model, as we discuss in detail in  Remark \ref{r.AMHolder}.

We would like to mention that the main initial motivation for this paper came from an observation that there are many parallel results on regularity of IDS of random Schr\"odinger operators, e.g. see \cite{L}, \cite{CKM}, \cite{Bo}, \cite{SVW}, \cite{HV}, and regularity of stationary measures of smooth random dynamical systems, e.g. see \cite{GKM}, \cite{M}. Theorems \ref{t.main2} and \ref{c:iid-Holder} provide an explanation of these similarities.\footnote{For Anderson Model, an explicit formula that relates the stationary measure of the projective Schr\"odinger cocycle and the IDS of the corresponding family of random Schr\"odinger operators, was provided in \cite{Sch} and \cite[Theorem 2.3]{ST}. Nevertheless, that formula does not immediately provide a connection between regularity of the stationary measure and regularity of the IDS.}

As a second application mentioned above,
we use Theorem \ref{t.main} to show that for the Anderson Model with ergodic background the IDS must be H\"older continuous, see Theorem \ref{t.IDS}. Regularity of the IDS is a classical question in the theory of ergodic Schr\"odinger operators. In general, it is known to be log-H\"older, see \cite{CS1} for the classical theorem and \cite{GK4} for a dynamical proof that also allows a non-linear generalization of the Craig and Simon's result. For specific classes of ergodic operators, a better regularity of the IDS can be shown in many cases. The case of Anderson Model (random iid potentials) will be  discussed in  Remark \ref{r.AMHolder}  below. For the Fibonacci Hamiltonian, the IDS must be H\"older continuous \cite{DG1}, while it is not always the case for operators with Sturmian potentials, see \cite{Mun} for details. H\"older continuity of the IDS for some types of quasiperiodic potentials was established in \cite{AJ}, \cite{GoS}, \cite{HA}, \cite{GJ}, while it is also known that it does not always hold in this case \cite{ALSZ}.  Theorem \ref{t.IDS} adds another series of examples to this list.




\subsection*{Plan of the paper}

 In Section \ref{s.main} we provide the setting and the formal statements of all three main results.  Section \ref{ss.1}  deals with the new formula for the increment of the rotation number of a one-parameter family of cocycles with the circle diffeomorphisms on the fiber, see Theorem \ref{t.main}. Section \ref{ss.2} is devoted to rotation numbers of  random dynamical systems on the circle, see Theorems \ref{t.main2} and \ref{c:iid-Holder}. In Section \ref{ss.3} we formulate the result on H\"older continuity of the IDS for Anderson Model with ergodic background potential, see Theorem \ref{t.IDS}.


 In Section~\ref{s.proofs3}, we prove the main result, Theorem~\ref{t.main}.
 Section~\ref{s.proofs} begins with the proofs of~Theorems~\ref{t.main2} and~\ref{c:iid-Holder} (see Section~\ref{s.T2}). We then discuss a possible alternative approach in Section~\ref{ss.alt}, and present an example showing that the assumption of absence of invariant measure in Theorem~\ref{c:iid-Holder} cannot be removed (see Section~\ref{ss.example}). Finally, in Section \ref{s.backgroundAM} we prove Theorem~\ref{t.IDS}.

\section{Main results}\label{s.main}

Let $(X, \mu)$ be a measure space, and  $T:X\to X$ be an invertible ergodic transformation, preserving the  measure~$\mgr$. Assume that we are given a skew product map $F_E$ over this transformation with the circle as the fiber, continuously depending on parameter $E\in J$, where  $J\subset \R$ is a compact interval:
\[
F_E: X\times \Sc \to X\times \Sc, \quad  (x,y) \mapsto (T(x), f_{E,x}(y)),
\]
where all the maps $f_{E,x}: \Sc\to\Sc$ are orientation-preserving homeomorphisms of the circle. Each map $f_{E,x}$ can be lifted from the circle $\Sc=\R/\Z$ to its universal cover; such a lift is defined up to a shift by an integer. We choose (and fix) the lifts
\[
\tf_{E,x} : \R \to \R
\]
of the maps $f_{E,x}$ continuously in $E$ and measurably in $x$; for instance, for a given point $y_0=0$ and a given $E_0\in J$ we can require $\tf_{E_0,x}(y_0)\in [0,1)$, then extending to other parameter values by continuity. Denote by
\begin{equation}\label{eq:S-def}
S_{n; E,x}(y) = (\tf_{E,T^{n-1}(x)} \circ \dots \circ \tf_{E,T(x)} \circ \tf_{E,x})(\ty) - \ty
\end{equation}
the shift of the lift $\ty\in \R$ of the initial point $y\in \Sc$ after $n$ iterations of lifted maps, corresponding to the parameter~$E$.
We can now define the rotation  number:
\begin{prop}\label{p.existence}
For every $E$ there exists a constant $\rh(E)$ such that for $\mgr$-a.e. $x\in X$ and for every $y \in \Sc$ one has
\begin{equation}\label{eq:rho-def}
\frac{1}{n} S_{n; E,x}(y)  \to \rh(E), \quad n\to\infty.
\end{equation}
\end{prop}

Proposition \ref{p.existence} is a consequence of the Birkhoff Ergodic Theorem, e.g. see \cite[Proposition A.1]{GK} or \cite{LL}. 

\begin{remark}
The exact value of the fibered rotation number does not have much meaning, since it depends on the choice of the lifts. Indeed, for each point $x\in X$ the lift $\tf_{E,x}$ of the map $f_{E,x}$ is defined up to an integer $k(x)$ that can depend on the point $x\in X$. For a different choice of lifts $\{\tf_{E,x}+k(x)\}$ the fibered rotation number will shift by $\int_X k(x)d\mu(x)$, which, for a measurable function $k(x)$, can take any value. One way to assign a specific value to the fibered rotation number would be to equip $X$ with some additional structure, e.g. to assume that $X$ is a compact metric space, and to require the lifts to depend continuously on the point $x\in X$ in the base, e.g. see \cite[Section 4.5]{AK} as an example of a setting  where this approach is convenient. Alternatively, and this is the approach we will use in this paper, one can notice that an increment of the fibered rotation number $\rho(E_2)-\rho(E_1)$ does not depend on the choice of the lifts, and therefore all the properties of the fibered rotation number that can be expressed only in terms of such increments (being constant on a given sub-interval of $J$, being Lipschitz or H\"older function, or having any other specific modulus of continuity, etc) are well defined.
\end{remark}

\subsection{Explicit formula for the fibered rotation number}\label{ss.1}

Note that for every $E$ there exists at least one invariant measure $\msp_E$ of the skew product $F_E$ that projects to $\mgr$ under the first coordinate projection $\pi:X\times \Sc\to X$, e.g.  see \cite[Theorem 1.5.10]{Ar}. 

Let $\msp_{E,x}$ be the fibered conditional measures (defined for $\mgr$-a.e. $x\in X$):
\[
\msp_E(A) = \int_X \msp_{E, x}(A \cap \pi^{-1}(\{x\})) \, d\mgr(x).
\]
Note that due to the invariance of $\nu$ (and invertibility of the base map $T$) we have
\[
(f_{E,x})_* \msp_{E, x} = \msp_{E, T(x)}
\]
for $\mgr$-a.e. $x\in X$.

Now, choose and fix any two parameter values $E_1,E_2\in J$ and any two invariant measures $\msp_{E_1}$, $\msp_{E_2}$ of the corresponding skew products $F_{E_1}, F_{E_2}$. Let $\msp_{E_1,x}$, $\msp_{E_2,x}$ be the corresponding conditional (probability) measures (defined for $\mgr$-a.e. $x\in X$). For a probability measure $\msp$ on the circle let $\widetilde{\msp}$ to be its (infinite total mass) lift on $\R$, and let for $a,b\in \R$
\begin{equation}\label{e.Phi}
\Phi_{\msp}(a,b):= \int_{\R} (\Ind_{[a,+\infty)} - \Ind_{[b,+\infty)}) \, d\widetilde{\msp} =
\begin{cases}
\widetilde{\msp} ([a,b)), & a< b \\
0, & a=b \\
-\widetilde{\msp} ([b,a)), & b<a
\end{cases}
\end{equation}
be the result of measuring the ``oriented half-interval'' from $a$ to $b$ with the lifted measure $\widetilde{\msp}$.
Finally, define
\begin{equation}\label{eq:g-def}
\shift_{E_1,E_2,x;\msp}(y):= \Phi_{\msp} (\tf_{E_1,x}(\ty),\tf_{E_2,x}(\ty)),
\end{equation}
where $\ty\in \R$ is a lift of $y\in \Sc$; it is easy to see that the right hand side of~\eqref{eq:g-def} does not depend on the choice of this lift.

Our first main result is then the following formula:
\begin{theorem}\label{t.main}
For any $E_1,E_2\in J$, and any corresponding invariant measures $\msp_{E_1},\msp_{E_2}$ we have
\begin{equation}\label{eq.main}
\rh(E_2)-\rh(E_1) = \int_{X\times \Sc} \shift_{E_1,E_2,x;\msp_{E_2,T(x)}}(y)\, d\msp_{E_1}(x,y),
\end{equation}
where $\shift_{E_1,E_2,x,\nu}(y)$ is given by~\eqref{eq:g-def}.
\end{theorem}

\begin{figure}[!ht]
\includegraphics[scale=0.8]{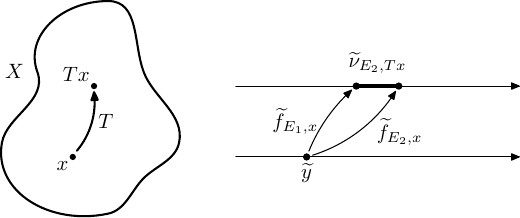}
\caption{The expression under the integral in~\eqref{eq.main}: the (signed)  $\widetilde{\msp}_{E_2,Tx}$-measure of the half-interval drawn in bold.}\label{f:shift}
\end{figure}

%
%

\subsection{The rotation number for random dynamical system on the circle}\label{ss.2}

Random dynamical systems on the circle can have highly non-trivial dynamics, e.g. see \cite{A, KN, GS, LL}. The statement of Theorem \ref{t.main} applied to a one-parameter  i.i.d. random dynamical systems on the circle (that can be seen as a one parameter family of skew products over the Bernoulli shift) can be essentially simplified.  Namely, assume that $g_{E,\omega}$ are orientation-preserving circle homeomorphisms, depending continuously on the parameter $E\in J \subset \R$ and measurably on $\omega \in \Omega$, where $(\Omega, \Prob)$ is some probability space. For every $E\in J$, choose a stationary measure $\nu_E^+$, as well as a stationary measure~$\nu_E^-$ for the inverse maps:
\[
\nu_E^+ = \E (g_{E,\omega})_* \nu_E^+, \quad \nu_E^- = \E (g_{E,\omega}^{-1})_* \nu_E^-.
\]
Notice that the choice of the stationary measures $\nu_E^+$ and $\nu_E^-$ in general is not unique. The result below is independent of this choice.

\begin{theorem}\label{t.main2}
For any $E_1,E_2\in J$, and any corresponding stationary measures  $\nu_{E_1}^+$ and $\nu_{E_2}^-$ we have
\[
\rho(E_2)-\rho(E_1)
= \E \int_{\Sc} \Phi_{\nu_{E_2}^{-}} (\widetilde{g}_{E_1,\omega}(y),\widetilde{g}_{E_2,\omega}(y)) \, d\nu_{E_1}^+(y),
\]
where $\Phi_{\msp}$ is given by~\eqref{e.Phi}.
\end{theorem}

To simplify the statement of Theorem \ref{t.main2}, assume further that for any fixed $\omega\in \Omega, \, y\in \Sc$ the function $g_{\cdot, \omega}(y)$ of the parameter $E\in J$ is monotone increasing and its image is a proper sub-arc of the circle. Then, we have the following formula:

\begin{coro}\label{c:iid}
In the setting above, for any $E_1<E_2$ one has
\[
\rho(E_2)-\rho(E_1)
= \E \int_{\Sc} \nu_{E_2}^{-} (I_{E_1,E_2;\omega}(y)) \, d\nu_{E_1}^+(y)
= \E (\nu_{E_1}^+ \times \nu_{E_2}^{-})  (K_{E_1,E_2;\omega}),
\]
where $I_{E_1,E_2;\omega}(y)=[g_{E_1,\omega}(y),g_{E_2,\omega}(y))$ and $K_{E_1,E_2;\omega}$ is the projection from~$\R^2$ on~$(\Sc)^2$ of the set
\[
\widetilde{K}_{E_1,E_2;\omega} :=\{(y,z) \mid \widetilde{g}_{E_1,\omega}(y)\le z <\widetilde{g}_{E_2,\omega}(y) \}.
\]
\end{coro}

\begin{figure}[!ht]
\includegraphics{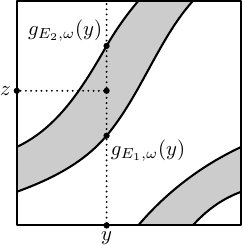}
\caption{Set $K_{E_1,E_2;\omega}$ and a point $(y,z)$ inside it.}
\end{figure}


Theorem \ref{t.main2} allows to estimate  the regularity of the fibered rotation number as a function of the parameter.

\begin{theorem}\label{c:iid-Holder}
Assume that in the setting of Theorem \ref{t.main2} or Corollary~\ref{c:iid}   the maps $g_{E,\omega}$ are diffeomorphisms of the circle that additionally satisfy the following properties:
\begin{itemize}
\item For any $E\in J$ there is no measure $\msp$ on the circle that would be $\Prob$-almost surely $g_{E,\omega}$-invariant.
\item There exists a constant $C_0>0$ such that for all $\omega\in \Omega$, $y\in S^1$, and $E\in J$ we have
$$
\left|\frac{\partial \tg_{E, \omega}(y)}{\partial E}\right|\le C_0.
$$
\item There are constants $C_1>0$ and $\gamma>0$ such that for every $E\in J$ we have
$$\int_{\Omega} (\mathfrak{L}(g_{E,\omega}))^{\gamma}  d\P(\omega)<C_1, $$
where
$$
\mathfrak{L}(g_{E,\omega})=\max_{y\in S^1}\max\left(|g'_{E, \omega}(y)|,  |(g^{-1}_{E, \omega})'(y)|\right).
$$
\end{itemize}
Then the fibered rotation number $\rho(E)$ is H\"older continuous: there exist $C,\alpha>0$ such that for any $E_1,E_2\in J$
\[
|\rho(E_2)-\rho(E_1)| \le C |E_2-E_1|^{\alpha}.
\]
\end{theorem}

The proof of this statement is a straightforward combination of Theorem \ref{t.main2} and the results on H\"older continuity of the stationary measures from \cite{GKM}; it is presented in Section \ref{s.proofs}.

\begin{remark}\label{r.AMHolder}
Theorem \ref{c:iid-Holder} can be considered as a non-linear analog of a well known statement from spectral theory of discrete Schr\"odinger operators. Namely, consider a family of random Schr\"odinger operators $H:\ell^2(\Z)\to \ell^2(\Z)$,
$$
[H \psi](n) = \psi(n+1) + \psi(n-1) + W(n) \psi(n),
$$
where $\{W(n)\}$ is a sequence of iid random variables distributed with respect to a probability measure $\xi$ on $\R$ that have some finite moments, i.e.  for some $\gamma>0$ we have $\int_{\mathbb{R}}|x|^\gamma d\xi<\infty$. This family of operators is known as 1D Anderson Model, and it was heavily studied,  see \cite[Chapter 5]{DF2}, \cite{AW}, and references therein.  One of the spectral characteristics of an ergodic family of Shr\"odinger operators, besides spectrum and spectral measures, is the integrated density of states (IDS). The IDS is the distribution function of the density of states measure $N_W$, which is the limiting distribution of the eigenvalues of restrictions $H|_{[1,L]}$ of the operator to finite intervals (which explains the terminology), with Dirichlet boundary condition:
$$
N_W(a,b) = \lim_{L \to \infty} \frac{1}{L} \# \big\{ \text{eigenvalues of } H|_{[1,L]} \text{ that lie in } (a,b) \big\},
$$
for any interval $(a,b)\subseteq \R$. See \cite{DF1} for different equivalent definitions of the IDS and a detailed discussion of its properties.
It is well known that the IDS in the Anderson Model must be (locally) H\"older continuous, see \cite{L}, \cite{CKM}, \cite{Bo}, \cite{SVW}, \cite{HV}. For a closely related statement on H\"older continuity of Lyapunov exponents in Anderson Model or, more generally, random projective cocycles,  see \cite{DGr1, DGr2, L2}. At the same time, if one considers a projectivized Schr\"odinger cocycle that corresponds to this family of ergodic Schr\"odinger operators parameterized by the energy $E$ (see \cite{DF1} for a definition of a Schr\"odinger cocycle and results that connect its dynamics with spectral properties of the operator), then a natural choice of lifts lead to the following relation between the rotation number of the cocycle and the IDS:
$$
\mathrm{IDS}(E)=1-2\rho(E),
$$
e.g. see \cite[Theorem 1.6]{DS}.
Therefore, H\"older continuity of the IDS in 1D Anderson Model can be considered as a partial case of Theorem \ref{c:iid-Holder}.
\end{remark}

\begin{remark}
In Theorem \ref{c:iid-Holder}, smoothness of the maps and smooth dependence on the parameter are  used only to extract Lipschitz properties of those maps. One can rewrite the statement of  Theorem \ref{c:iid-Holder} requesting only that the maps $g_{E, \omega}$ are Lipschitz in the parameter $E$ with uniform Lipschitz constant, and that the moment condition holds for ${\mathfrak{L}}(g_{E, \omega})=\max(\text{\rm Lip}(g_{E, \omega}), \text{\rm Lip}(g_{E, \omega}^{-1})).$
\end{remark}

\begin{remark}
The condition on absence of invariant measures in  Theorem \ref{c:iid-Holder} cannot be removed. We provide an explicit example to demonstrate it in Section \ref{ss.example}.
\end{remark}

\begin{remark}
It is interesting to observe similarities between the regularity results on Lyapunov exponents and the fibered rotation number. In the case of Schr\"odinger cocycles those similarities can be explained by the so called Thouless formula \cite{T, AS, CS1}. In the case of more general projective cocycles of special form a generalized Thouless formula \cite{BCDFK} is also available. We are not aware of a simple non-linear analog of Thouless formula that would explain, for example, the similarities between \cite{GK4} and \cite{TV}.
\end{remark}

\begin{remark}\label{r.DM}
After this work was completed, we learned about the recent preprint \cite{BM}; Proposition V from there claims that a stationary measure of a mostly contracting random map depends H\"older continuously on a parameter. This should allow to provide an alternative proof of Theorem \ref{c:iid-Holder} in this case.
\end{remark}

\subsection{Regularity of the IDS of the Anderson Model with ergodic background potential}\label{ss.3}

 The formula provided by Theorem \ref{t.main} can be used to extend the results on H\"older continuity of the IDS in Anderson Model, as discussed in Remark \ref{r.AMHolder}  above, to a larger class of ergodic Schr\"odinger operators with iid random noise.

Given a compact metric space $M$, a homeomorphism $G : M \to M$, an ergodic Borel probability measure $m$ with full topological support, $\supp m = M$, and a sampling function $\varphi \in C(X,\R)$, we generate potentials
$$
V_x(n) = \varphi(G^n x), \; x \in M, \; n \in \Z
$$
and Schr\"odinger operators
$$
[H_x \psi](n) = \psi(n+1) + \psi(n-1) + V_x(n) \psi(n)
$$
on $\ell^2 (\Z)$. The spectrum of $H_x$, denoted by $\sigma(H_x)$, is almost surely independent of~$x$, e.g. see \cite{Pa} or \cite{DF1}. That is, there is a compact set $\Sigma_0$ such that
\begin{equation}
\Sigma_0 = \sigma(H_x) \text{ for $\mu$-almost every } x \in X.
\end{equation}
The random perturbation is given by
$$
W_\omega(n) = {\omega}_n, \; \bar\omega \in \Omega^\Z, \; n \in \Z,
$$
where $\Omega = \supp \xi$ and $\xi$ is a 
probability measure on $\R$ with topological support
satisfying
$
\# \Omega \ge 2.
$
We will also assume that $\xi$ has finite exponential moment, i.e. for some $\gamma>0$ we have $\int_{\mathbb{R}}|x|^\gamma d\xi<\infty$.
\begin{theorem}\label{t.IDS}
In the setting above, the integrated density of states of the family of operators $\{H_x+W_\omega\}_{x\in X, \omega\in \Omega}$ is H\"older continuous on any compact interval $J\subseteq\R$ of energies.
\end{theorem}

While  most likely Theorem  \ref{t.IDS} can be obtained by some spectral methods, it does not seem to be currently available in the spectral theory literature. In Section~\ref{s.backgroundAM} we provide a purely dynamical proof of Theorem \ref{t.IDS} based on Theorem \ref{t.main} and the results from \cite{GKM}.







\section{Proof of the explicit formula for the rotation number}\label{s.proofs3}

The proof of our main result is based on an (exact) formula for the translation number of a skew product, Theorem~\ref{p:formula} below, that could be of independent interest. Prior to stating it, let us consider a basic example of the rotation number of a single homeomorphism. Namely, let $g:\Sc\to \Sc$ be a homeomorphism without fixed points; then it has a lift $\widetilde{g}$ satisfying $\ty<\widetilde{g}(\ty)<\ty+1$ for all $\ty\in \R$. Let $\msp_1,\msp_2$ be two probability measures on the circle, and consider the integral
\begin{equation}\label{eq:I-def}
\mathcal{I}(\msp_1,\msp_2)=\int_{\Sc} \msp_2(I(y)) \, d\msp_1(y) = (\msp_1\times \msp_2)(K),
\end{equation}
where $I(y)=[y,g(y))$ and $K$ is the projection on $(\Sc)^2$ of the set $\{(\ty,\widetilde{z}) \mid \ty<\widetilde{z}<\widetilde{g}(\ty)\}$ on the plane.

\begin{prop}
If at least one of two measures $\msp_1,\msp_2$ is $g$-invariant, then $\mathcal{I}(\msp_1,\msp_2)=\rho(\widetilde{g})$.
\end{prop}
\begin{proof}
If $\msp_1$ is $g$-invariant, then it is a corollary of Birkhoff ergodic theorem. Namely, we have
\[
\widetilde{\msp}_2([\ty, \widetilde{g}^n(\ty))) = \sum_{j=0}^{n-1} \msp_2(I(g^j(y));
\]
after division by $n$ and passing to the limit the left hand side converges to the rotation number $\rho(\widetilde{g})$, while the right hand side by Birkhoff ergodic theorem converges to $\mathcal{I}(\msp_1,\msp_2)$.

If $\msp_2$ is $g$-invariant, then actually for any $y$ one has $\msp_2(I(y))=\rho(\widetilde{g})$, so in~\eqref{eq:I-def} we're integrating a constant function.
\end{proof}

Now, consider a skew product $F(x,y)=(Tx,f_x(y))$ and its lift $\widetilde{F}(x,\ty)=(Tx,\tf_x(y))$. Given two  measures $\msp_1,\msp_2$ on $X\times \Sc$ that project to $\mgr$ on the first coordinate, let us define the translation value associated to these measures. To do so, for any two compactly supported probability  measures $m_1,m_2$ on $\R$ and a measure $\msp$ on the circle, define the {signed shift} of measure $m_2$ with respect to measure $m_1$, {measured with}~$\nu$ by
\[
\Phi_{\msp}(m_1,m_2) :=\int \left(m_1((-\infty,y])-m_2((-\infty,y])\right) \, d\widetilde{\msp}(y),
\]
where $\widetilde{\msp}$ is the infinite mass lift of $\msp$, i.e. a unique infinite measure of $\R$ such that for any $c\in\R$, a projection of the restriction $\msp|_{[c, c+1)}$ to $\Sc$ gives $\msp$. Note that for the point masses $m_1=\delta_a$, $m_2=\delta_b$, this shift is equal to $\Phi_{\msp}(a,b)$ defined by~\eqref{e.Phi}. In general, $\Phi_{\msp}(m_1,m_2)$  is equal to the result of the averaging of $\Phi_{\msp}(a,b)$:
\begin{lemma}\label{l.shift}
    For any probability measure~$m$ on pairs $(a,b)$ with marginals $m_1$ and $m_2$, one has
    \[
        \Phi_{\msp}(m_1,m_2) = \int \Phi_{\msp}(a,b) \, dm(a,b).
    \]
\end{lemma}
\begin{proof}
    Consider a function $D(a,b;y)$ on $\R^3$, given by
    \[
        D(a,b;y)=\Ind_{[a,+\infty)}(y) - \Ind_{[b,+\infty)}(y).
    \]
    It then suffices to calculate the integral of $D(a,b;y)$ with respect to $m\times \widetilde{\msp}$ with two different orders of integration. Indeed, on one hand,
    \[
        \int D(a,b;y) \, dm(a,b) \, d\widetilde{\msp}(y) = \int \Phi_{\msp}(a,b) \, dm(a,b),
    \]
    while on the other, integrating over $m(a,b)$ first gives
    \[
        \int D(a,b;y) \, dm(a,b) = m_1((-\infty,y])-m_2((-\infty,y]),
    \]
    and integration over $\widetilde{\msp}$ concludes the proof.
\end{proof}

Next, let $\{\msp_{1,x}\}_{x\in X}$ be the decomposition of the measure $\msp_1$ into the conditional measures  on the circle, and let $\{\msp'_{1,x}\}$ be a family of compactly supported  probability measures on $\R$ with uniformly bounded supports such that $\msp'_{1,x}$ 
projects to $\msp_{1,x}$. We define the translation value
\begin{equation}\label{eq:mT-def}
    \mathcal{T}(\tF;\msp_1,\msp_2):=\int_X  \Phi_{\msp_{2,Tx}} (\msp'_{1,Tx},(\widetilde{f}_x)_* \msp'_{1,x}) \, d\mgr(x).
\end{equation}

\begin{remark}
It is not hard to see that the quantity $\mathcal{T}(\tF;\msp_1,\msp_2)$ is well defined, i.e.  does not depend on the choice of the lifts $\msp'_{1,x}$.
\end{remark}

\begin{example}
{If $\msp_1$ is a measure supported on a section $y=\gamma(x)$, and $\msp_2=\mgr\times \Leb_{\,\Sc}$, then $\mathcal{T}(\tF;\msp_1,\msp_2)$ is exactly the average shift between the lift of the section and its $\widetilde{F}$-image}, see Fig.\,\ref{f:sections}.
\end{example}

\begin{figure}[!ht]
\includegraphics[scale=0.8]{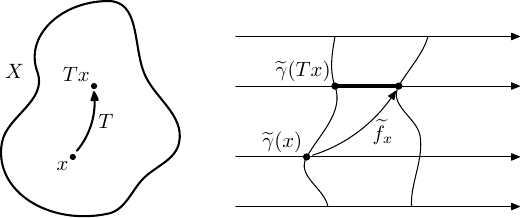}
\caption{The case of a measure $\nu_1$ supported on a section $\gamma$: the translation value $\mathcal{T}$ is equal to the average measure of the half-open interval between the lift of the section $\gamma$ and its image $\widetilde{F}\gamma$.} \label{f:sections}
\end{figure}

\begin{lemma}\label{l:t-limit}
For any probability measures $\msp_1,\msp_2$ on $\Sc$, one has
\[
\lim_{n\to\infty} \frac{1}{n} \mathcal{T}(\tF^n;\msp_1,\msp_2) = \rh(\tF)
\]
\end{lemma}
\begin{proof}
{Take the lifts $\msp'_{1,x}$ to be supported on $[0,1]$ for all $x\in X$. Then for any $n$ and every $x\in X$ the push-forward measure
\[
(\tf_{T^{n-1}x}\circ \dots \circ \tf_{Tx}\circ \tf_x)_* \msp'_{1,x}
\]
is supported on the interval
\[
(\tf_{T^{n-1}x}\circ \dots \circ \tf_{Tx}\circ \tf_x)([0,1]).
\]
As the $d\widetilde{\msp}_{2,x}$-measure of every half-open interval of length~$1$ is exactly equal to~$1$, one has
\[
\left| \Phi_{\msp_{2,T^{n}x}} \left(\msp'_{1,T^nx},(\tf_{T^{n-1}x}\circ \dots \circ \tf_{Tx}\circ \tf_x)_* \msp'_{1,x}\right) - S_{n; E, x}(0)  \right| \le 2.
\]
Integrating over~$x\in X$, we get
\[
\left| \mathcal{T}(\tF^n;\msp_1,\msp_2) -\int_X S_{n; E, x}(0) \, d\mgr(x) \right| \le 2;
\]
dividing by $n$ and passing to the limit using Proposition~\ref{p.existence}, we complete the proof.
}
\end{proof}

\begin{lemma}\label{l:cocycle}
The translation value satisfies a cocycle relation:
    \begin{equation}\label{eq:cocycle}
        \mathcal{T}(\tF^n;\msp_1,\msp_2)=\mathcal{T}(\tF;\msp_1,\msp_2)+\mathcal{T}(\tF^{n-1};F_*\msp_1,\msp_2)
    \end{equation}
\end{lemma}
\begin{proof}
    Let us first rewrite the definition~\eqref{eq:mT-def}. Namely, let $\msp_1'$ be the measure on $X\times \R$ such that its conditional measure on a.e. fiber $\{x\}\times \R$ is $\msp'_{1,x}$. Then, the change of variables $y=Tx$ transforms~\eqref{eq:mT-def} into
    \[
        \int_X  \Phi_{\msp_{2,y}} (\msp'_{1,y},(\widetilde{F}_* \msp'_{1})_y) \, d\mgr(y),
    \]
    where $(\cdot)_y$ stays for the conditional measure on the corresponding fiber. We then get
    \begin{multline*}
        \mathcal{T}(\tF^n;\msp_1,\msp_2)=
        \int_X  \Phi_{\msp_{2,y}} (\msp'_{1,y},(\widetilde{F}^n_* \msp'_{1})_y) \, d\mgr(y)
        \\
        = \int_X  \Phi_{\msp_{2,y}} (\msp'_{1,y},(\widetilde{F}_* \msp'_{1})_y) \,
        d\mgr(y) +  \int_X  \Phi_{\msp_{2,y}} ((\widetilde{F}_* \msp'_{1})_y,(\widetilde{F}^n_* \msp'_{1})_y) \,
        d\mgr(y)
        \\
        =\mathcal{T}(\tF;\msp_1,\msp_2)+\mathcal{T}(\tF^{n-1};F_*\msp_1,\msp_2),
    \end{multline*}
    where for the second equality we have used the additivity relation
    \begin{equation}\label{eq:Phi-additive}
        \Phi_{\nu}(m_1,m_3) = \Phi_{\nu}(m_1,m_2) +
        \Phi_{\nu}(m_2,m_3).
    \end{equation}
\end{proof}

\begin{theorem}\label{p:formula}
If at least one of the measures $\msp_1$, $\msp_2$ is $F$-invariant, then the translation number is exactly equal to the translation value:
\[
\rh(\tF) = \mathcal{T}(\tF;\msp_1,\msp_2).
\]
\end{theorem}

\begin{proof}
Take an arbitrary $n$, and let us use the cocycle relation~\eqref{eq:cocycle} from Lemma~\ref{l:cocycle}. Namely, if $\msp_1$ is invariant, one has
\[
\mathcal{T}(\tF^n;\msp_1,\msp_2)=\mathcal{T}(\tF;\msp_1,\msp_2)+\mathcal{T}(\tF^{n-1};\msp_1,\msp_2),
\]
so by induction
\[
\mathcal{T}(\tF^n;\msp_1,\msp_2)=n\mathcal{T}(\tF;\msp_1,\msp_2).
\]
Dividing by $n$ and applying Lemma~\ref{l:t-limit}, we get the desired conclusion.

On the other hand, if $\msp_2$ is invariant, we can rewrite the second summand in the right hand side of~\eqref{eq:cocycle} by applying $F^{-1}$:
\[
\mathcal{T}(\tF^n;\msp_1,\msp_2)=\mathcal{T}(\tF;\msp_1,\msp_2)+\mathcal{T}(\tF^{n-1};\msp_1, F_*^{-1}\msp_2).
\]
Thus in this case we also get by induction
\[
\mathcal{T}(\tF^n;\msp_1,\msp_2)= n\mathcal{T}(\tF;\msp_1,\msp_2);
\]
again, dividing by $n$ and passing to the limit, we obtain the desired relation.
\end{proof}

Theorem~\ref{t.main} now follows from Theorem~\ref{p:formula}:
\begin{proof}[Proof of Theorem~\ref{t.main}]
As measures $\msp_{E_1}$ and $\msp_{E_2}$ are invariant for $F_{E_1}$ and $F_{E_2}$ respectively, from Theorem~\ref{p:formula} we have
\[
\rho(E_1)= \mathcal{T}(\tF_{E_1};\msp_{E_1},\msp_{E_2}), \quad
\rho(E_2)= \mathcal{T}(\tF_{E_2};\msp_{E_1},\msp_{E_2}).
\]
Using the definition of the translation value $\mathcal{T}$, we get
\begin{multline}\label{eq:diff-rho}
\rho(E_2)-\rho(E_1)= \mathcal{T}(\tF_{E_2};\msp_{E_1},\msp_{E_2})- \mathcal{T}(\tF_{E_1};\msp_{E_1},\msp_{E_2})
\\
=\int_X  \left( \Phi_{\msp_{E_2,Tx}} (\msp'_{E_1,Tx},(\widetilde{f}_{E_2,x})_* \msp'_{E_1,x}) -
\Phi_{\msp_{E_2,Tx}} (\msp'_{E_1,Tx},(\widetilde{f}_{E_1,x})_* \msp'_{E_1,x})
\right) \,
d\mgr(x)
\\=\int_X  \left( \Phi_{\msp_{E_2,Tx}} (
(\widetilde{f}_{E_1,x})_* \msp'_{E_1,x},(\widetilde{f}_{E_2,x})_* \msp'_{E_1,x})
\right) \,
d\mgr(x)
\end{multline}
where we have used the additivity relation~\eqref{eq:Phi-additive}.
Finally, using Lemma \ref{l.shift} the value under the integral in the right hand side of~\eqref{eq:diff-rho}  can be written as
\begin{multline*}
\Phi_{\msp_{E_2,Tx}} (
(\widetilde{f}_{E_1,x})_* \msp'_{E_1,x},(\widetilde{f}_{E_2,x})_* \msp'_{E_1,x}) =
\int_{\R} \Phi_{\msp_{E_2,Tx}} (\widetilde{f}_{E_1,x}(\ty),\widetilde{f}_{E_2,x}(\ty)) \, d\msp'_{E_1,x}(\ty)
\\
= \int_{\Sc} \theta_{E_1,E_2;x,\msp_{E_2,T_x}}(y)  \, d\msp_{E_1,x}(y).
\end{multline*}
Substituting it to~\eqref{eq:diff-rho}, and recalling that $\msp_{E_1,x}$ are the conditional measures of $\msp_{E_1}$ w.r.t.~$\mgr$, we obtain the desired
\begin{multline*}
\rho(E_2)-\rho(E_1)=\int_X  \left( \int_{\Sc} \theta_{E_1,E_2;x,\msp_{E_2,T_x}}(y)  \, d\msp_{E_1,x}(y)
\right) \,
d\mgr(x)
\\
=\int_{X\times \Sc}   \theta_{E_1,E_2;x,\msp_{E_2,T_x}}(y) \,
d\msp_{E_1}(x,y).
\end{multline*}
\end{proof}

\section{The fibered rotation number for random dynamical systems}\label{s.proofs}

Here we provide the proofs of Theorems \ref{t.main2} and \ref{c:iid-Holder} and discuss an alternative argument that could a better intuition behind the statement of Theorems \ref{t.main2}. Finally, in the  section \ref{ss.example} we give an example of a one parameter family of random dynamical systems on the circle that shows that the assumption of absence of invariant measures in Theorem \ref{c:iid-Holder}  cannot be removed.

\subsection{Proof of Theorems \ref{t.main2} and \ref{c:iid-Holder}}\label{s.T2}

\begin{proof}[Proof of Theorem \ref{t.main2}]
Theorem \ref{t.main2} can be deduced from Theorem \ref{t.main}. Indeed, for given $E_1, E_2$ choose the corresponding stationary measures $\nu_{E_1}^+$ and $\nu_{E_2}^-$. Existence of the stationary measures is guaranteed by \cite[Theorem 1.7.5 and Theorem 2.1.8]{Ar}.

Denote $X=\Omega^{\Z}$,  $x=\ldots \omega_{-n}\ldots\omega_0\omega_1\ldots\omega_n\ldots\in X$, and let $\sigma:\Omega^{\Z}\to \Omega^{\Z}$ be the left shift, and $\mu=\P^{\Z}$. Also, set $X=\Omega^{\Z}=\Omega^-\times \Omega_0\times \Omega^+$, where
$$
\Omega^-=\{(\ldots\omega_{-n}\ldots \omega_{-1})\ |\ \omega_i\in \Omega\}, \Omega^+=\{(\omega_1 \ldots\omega_{n}\ldots)\ |\ \omega_i\in \Omega\}, \Omega_0=\Omega.
$$
For any $x\in X, x=\ldots \omega_{-n}\ldots\omega_0\omega_1\ldots\omega_n\ldots,$ denote $x^-=(\ldots\omega_{-n}\ldots \omega_{-1})\in \Omega^-$, $x^+=(\omega_1 \ldots\omega_{n}\ldots)\in \Omega^+$, $x_0=\omega_0\in \Omega_0.$

Consider the cocycle
$$
F_E:X\times\Sc\to X\times\Sc, F_E(x, y)=(\sigma x, g_{E, x_0}(y)),
$$
and denote by $\pi_{\Sc}:X\times \Sc \to \Sc$ a natural projection.
 There is a one-to-one correspondence between stationary measures and invariant measures of $F_E$ that ``depend only on the past'' (called {\it forward Markov measures} in \cite{Ar}), and stationary measures for inverse maps and invariant measures of $F_E$ that ``depend only on the future'' (called {\it backward Markov measures} in \cite{Ar}).  Namely, there are invariant measures $\eta_{E_1}^+$ and $\eta_{E_2}^-$ with $(\pi_{\Sc})_*\eta^+_{E_1}=\nu^+_{E_1}, \ (\pi_{\Sc})_*\eta^-_{E_2}=\nu^-_{E_2}$, and such that measures $\left\{\eta_{E_1, x}^+\right\}_{x\in X}$ in the decomposition of $\eta_{E_1}^+$ depend only on $x^-$, and measures $\left\{\eta_{E_2, x}^-\right\}_{x\in X}$ in the decomposition of $\eta_{E_2}^-$ depend only on $(x_0, x^+).$

Application of Theorem \ref{t.main} to the cocycle $F_E$ and measures $\eta_{E_1}^+$ and $\eta_{E_2}^-$ gives:
\begin{multline}
    \rho(E_2)-\rho(E_1)=\int_{X\times \Sc}\Phi_{\eta^-_{E_2, \sigma(x)}}(\tg_{E_1, x_0}(\ty), \tg_{E_2, x_0}(\ty))d\eta^+_{E_1}(x, y)=\\
    \int_X\int_{\Sc}\Phi_{\eta^-_{E_2, x^+}}(\tg_{E_1, x_0}(\ty), \tg_{E_2, x_0}(\ty))d\eta^+_{E_1, x^-}(y)d\P^{\N}(x^-) d\P^{\N}(x^+) d\P(x_0)=\\
    \int_{\Omega_0}\int_{\Sc}\Phi_{\nu^-_{E_2}}(\tg_{E_1, x_0}(\ty), \tg_{E_2, x_0}(\ty))d\nu^+_{E_1}(y)d\P(x_0)=\\
    \E \int_{\Sc} \Phi_{\nu_{E_2}^{-}} (\widetilde{g}_{E_1,\omega}(y),\widetilde{g}_{E_2,\omega}(y)) \, d\nu_{E_1}^+(y).
\end{multline}
\end{proof}

\begin{proof}[Proof of Theorem \ref{c:iid-Holder}]
To simplify the exposition, we provide the proof in the setting of Corollary \ref{c:iid}; the proof in a more general setting of Theorem \ref{t.main2} is analogous.

By Lipchitz continuity of the maps $f_{E, \omega}$ in parameter, there exists $C_0>0$ such that, in the notations of Corollary \ref{c:iid}, for any $\ty\in\R$, $\omega\in \Omega$, and $E_1, E_2\in J$ we have
$$
|I_{E_1, E_2, \omega}(y)|\le C_0|E_2-E_1|.
$$
The conditions (measure condition and moments condition) in Theorem \ref{c:iid-Holder} allow to apply Theorem 1.1 from \cite{GKM}, which implies that the stationary measures $\nu_{E_1}^+$ and $\nu_{E_2}^-$ are H\"older continuous, i.e. for some $\alpha>0,$ $C'>0$, and any interval $I\subseteq \Sc$ one has
$$
\nu_{E_1}^+(I)\le C'|I|^\alpha, \ \ \nu_{E_2}^-(I)\le C'|I|^\alpha.
$$
Therefore, by Theorem \ref{t.main2} we have
\begin{multline}
|\rho(E_2)-\rho(E_1)| \le\\
\E \int_{\Sc} \nu_{E_2}^{-} (I_{E_1,E_2;\omega}(y)) \, d\nu_{E_1}^+(y)\le
\E \int_{\Sc} C'|I_{E_1,E_2;\omega}(y)|^\alpha \, d\nu_{E_1}^+(y)\le \\
\E \int_{\Sc} C'C_0^\alpha |E_2-E_1|^\alpha \, d\nu_{E_1}^+(y)=C'C_0^\alpha|E_2-E_1|^\alpha=
C |E_2-E_1|^{\alpha},
\end{multline}
with $C=C'C_0^\alpha.$


\end{proof}

\subsection{Alternative construction}\label{ss.alt}

It can be interesting to notice that Corollary~\ref{c:iid} admits also a different 
interpretation from the random dynamical point of view. Here we give a ``hand waiving'' explanation of this alternative approach; by filling out the details it can be turned into an alternative  rigorous proof, at least under some extra assumptions on the random dynamical system.

Consider the increment between the long random images of a given initial point~$y_0$ as the parameter varies between $E_1,E_2$. It can be represented as a sum of the increments, when among $n$ composed maps only for one the parameter is varied:
\begin{multline}\label{eq:shift}
(\tg_{E_2,\omega_n}\circ \dots \circ \tg_{E_2,\omega_1})(y_0)-
(\tg_{E_1,\omega_n}\circ \dots \circ \tg_{E_1,\omega_1})(y_0)=
\\
=\sum_{j=1}^n \bigl[
(\tg_{E_2,\omega_n}\circ \dots \circ \tg_{E_2,\omega_{j+1}}\circ \tg_{E_2,\omega_j} \circ \tg_{E_1,\omega_{j-1}} \dots \circ \tg_{E_1,\omega_1})(y_0) -
\\
(\tg_{E_2,\omega_n}\circ \dots \circ \tg_{E_2,\omega_{j+1}}\circ \tg_{E_1,\omega_j} \circ \tg_{E_1,\omega_{j-1}} \dots \circ \tg_{E_1,\omega_1})(y_0)
\bigr]
\end{multline}

Now, the expectation of this increment is equal to $(\rho(E_2)-\rho(E_1))\cdot n+o(n)$. On the other hand, the expectation of increments in the right hand side of~\eqref{eq:shift} can be analysed separately for different $j$. Namely, under quite general assumptions on the system, a long random composition sends all the circle, except for a small neighbourhood of a (random) ``repelling'' point, to a small neighbourhood of a (random) ``attracting'' point, see \cite{A, KN}; this can be viewed as a non-linear analog of positivity of Lyapunov exponents in Furstenberg Theorem on random matrix products and its variations \cite{Fur, FK, FKif, GK2}. The ``attracting'' point is distributed w.r.t $\msp^+$, the ``repelling'' one w.r.t. $\msp^-$.

\begin{figure}[!ht]
\includegraphics[scale=0.8]{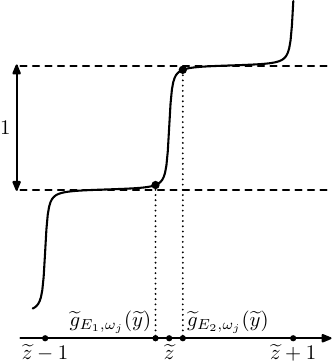}
\caption{The graph of the composition $\tg_{E_2,\omega_n}\circ \dots \circ \tg_{E_2,\omega_{j+1}}$ and the effect of changing the parameter at the $j$-th iteration.} \label{f:graphs}
\end{figure}

Next, let us calculate the expectations of individual increments in the right hand side of~\eqref{eq:shift}. When both $j$ and $n-j$ are large, the first $j-1$ maps send the initial point to a random point $y$, distributed (almost) w.r.t. the measure $\msp_{E_1}^+$, while the last $n-j$ define $\msp_{E_2}^-$-distributed ``repelling'' point $z$ on the circle. Then, if the arc $[\tg_{E_1,\omega_j}(y),\tg_{E_2,\omega_j}(y)]$ contains $z$, the image will be increased by (almost)~1, otherwise its increment is neglectable: see Fig.~\ref{f:graphs}. Hence, the expectation of this increment is asymptotically (as $j$ and $n-j$ become large) given by the probability of the ``arc covering the repelling point'' event, and that is given exactly by the expectation
\[
\E (\nu_{E_1}^+ \times \nu_{E_2}^{-})  (K_{E_1,E_2;\omega}).
\]

As it does not depend on $n$ or $j$, dividing by $n$ and passing to the limit would provide 
the conclusion of Corollary~\ref{c:iid}.


\subsection{Example}\label{ss.example}

Let us provide an explicit example of a one parameter family of random dynamical system on the circle with the fibered rotation number that is not H\"older. The example demonstrates that one cannot remove the condition of absence of invariant measure in Theorem \ref{c:iid-Holder}.

Let $f:S^1\to S^1$ be an orientation preserving  diffeomorphism of the circle that has exactly  two fixed points - one hyperbolic attractor and one hyperbolic repeller. Set $g_{E, 1}=R_E\circ f$ and $g_{E, 2}=R_E\circ f^{-1}$.

\begin{figure}[!ht]
\includegraphics[scale=0.9]{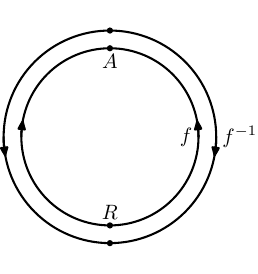} \quad
\includegraphics[scale=0.9]{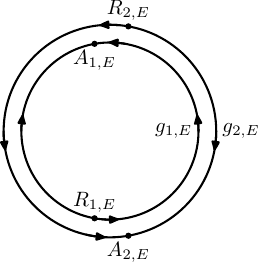} \quad
\includegraphics[scale=0.9]{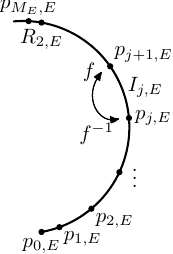}
\caption{Map $f$ and the system $\{g_{E,1}, g_{E,2}\}$ for $E>0$}\label{f:system}
\end{figure}

Consider a random dynamical system on the circle generated by $\{g_{E,1}, g_{E,2}\}$ applied randomly with equal probabilities. Choose the lifts $\tg_{E,1}, \tg_{E,2}$ such that $\tg_{E, 1}|_{E=0}(0)=\tg_{E, 2}|_{E=0}(0)=0.$ Then for the corresponding fibered rotation number we have $\rho(0)=0$.

At the same time, one can show
that for small positive $E>0$ one has
\begin{equation}\label{eq:rho}
\rho(E)\gtrsim \left(\log E^{-1}\right)^{-2}.
\end{equation}
Indeed, consider the attractors $A_{1,E}$ and $A_{2,E}$ and the repellers $R_{1,E}$ and $R_{2,E}$ of $g_{1,E}$  and $g_{2,E}$ respectively (see Fig.~\ref{f:system}). Note that at the arcs $[R_{2,E},A_{1,E}]$ and $[R_{1,E},A_{2,E}]$, the random iterations of an initial point can move only in the positive direction, and crossing these arcs requires in average a constant number of iterations. To study the crossing of the arc $[A_{2,E},R_{2,E}]$, consider the images
\[
p_{j,E}:=f^j(A_{2,E}), \quad j\ge 0,
\]
and let
\[
M_E:=\min\{j: p_{j,E}\in [R_{2,E},A_{1,E}]\}.
\]
Then $M_E \sim c \log E^{-1}$, as the points $R_{2,E}$ and $A_{2,E}$ are shifted by distance of order~$E$ from the attractor and the repeller of~$f$. Take the (half-open) arcs $I_{j,E}=[p_{j,E},p_{j+1,E})$; we then have
\[
I_{M_E,E}\subset [R_{2,E},A_{1,E}].
\]

Consider now the random sequence $j_{n,\omega}$ of indices, such that the arc $I_{j_{n,\omega},E}$ contains the random image
\begin{equation}\label{eq:g-iter}
g_{\omega_n,E}\circ\dots\circ g_{\omega_1,E}(x_0)
\end{equation}
of the initial point $x_0\in I_0$. Now, the comparison
\[
g_{1,E}(x)> f(x), \quad g_{2,E}(x)> f^{-1}(x)
\]
implies
\[
j_{n,\omega} \ge
\begin{cases}
    j_{n-1,\omega} + 1, & \omega_n=1,\\
    j_{n-1,\omega} - 1, & \omega_n=2.
\end{cases}
\]
The sequence of indices $j_{n,\omega}$ can thus be coupled with a $+1/-1$ random walk on the interval $[0,M_E]$ that is reset at $0$ if tries to go in the negatives: this random walk serves as a lower bound for~$j_{n,\omega}$. Next, the expected time for this random walk to reach $M_E$, starting at~$j=0$, can be easily calculated: namely, the expectation $T_{j,E}$ as a function of a starting index $j$ satisfies the difference equation
\[
T_{j,E} = 1+ \frac{T_{j-1,E}+T_{j+1,E}}{2}
\]
with the boundary conditions
\[
T_{M_E,E}=0, \quad T_{-1,E}=T_{0,E}.
\]
Solving it leads to
\[
T_{j,E}=M_E(M_E+1) - j(j+1),
\]
that has an asymptotics~$\sim M_E^2$ at $j=0$.

The expected number of steps for the random iterations~\eqref{eq:g-iter} to cross the arc $[A_{2,E},R_{2,E}]$ hence is also bounded from above by
\[
M_E^2 \sim c^2 (\log E^{-1})^2.
\]
The same applies to the expected number of iterates to cross~$[A_{1,E},R_{1,E}]$; hence, the fibered rotation number admits a lower bound by
\[
\frac{1}{2M_E^{2}} \sim \frac{1}{2c^{2}} (\log E^{-1})^{-2},
\]
thus proving~\eqref{eq:rho}. Therefore, the fibered rotation number cannot be H\"older continuous.

\begin{remark}
  While H\"older regularity of different structures (Lyapunov exponents, rotation number, stationary measures, topological conjugacies, holonomy maps along stable/unstable foliations, etc.) does appear in dynamics in many cases, it is not quite universal, as the example above shows. For some other examples of similar nature we refer the reader to~\cite{DKS} for a 
construction of a random projective cocycle with non-H\"older Lyapunov exponents\footnote{We refer the reader to \cite{DK1}, \cite{DK2}, \cite{V} for a general theory of Lyapunov exponents of linear cocycles.}, and to \cite[Appendix A]{GKM} for an example of a proximal and strongly irreducible random linear cocycle with finite polynomial but infinite  exponential moments such that the corresponding stationary measure on the projective space is not H\"older.
\end{remark}

\section{H\"older continuity of the IDS in Anderson Model:\\ proof of Theorem \ref{t.IDS}}\label{s.backgroundAM}


\begin{proof}[Proof of Theorem \ref{t.IDS}]
Transfer matrices that correspond to the ergodic family of Schr\"odinger operators
$$
[H_x \psi](n) = \psi(n+1) + \psi(n-1) + V_x(n) \psi(n)
$$
on $\ell^2 (\Z)$ have the form $\begin{pmatrix}
  E-\varphi(x) & -1 \\
  1 & 0 \\
\end{pmatrix}$. Their pojectivizations act on the projective line that we will identify with the circle $\Sc$; we denote these maps by $g_{E, x}:\Sc\to \Sc$. Therefore, the Schr\"odinger cocycle can be represented as a skew product with the circle fibers:
$$
T_E:M\times \Sc\to M\times \Sc, \ T_E(x, y)=(G(x), g_{E,x}(y)).
$$
Adding random noise leads to a random dynamical system with the phase space $M\times \Sc$,
$$
T_{E, \omega}:M\times \Sc\to M\times \Sc, \ T_{E, \omega}(x, y)=(G(x), g_{E,x, \omega}(y)),
$$
where $g_{E,x, \omega}$ is a projective map that corresponds to a random matrix  $\begin{pmatrix}
  E-\varphi(x) -W_\omega & -1 \\
  1 & 0 \\
\end{pmatrix}$, and $W_\omega$ is a random variable distributed with respect to the measure $\xi$. Denote by $\pi_M:M\times \Sc\to M$ the natural projection.  For each $E\in J$, there exits a (possibly not unique) stationary measure $\nu^+_E$ on $M\times \Sc$  such that $(\pi_M)_*\nu^+_E=m$. Similarly, there exists a stationary measure for inverse maps, $\nu^-_E$ with $(\pi_M)_*\nu^-_E=m$.

\begin{lemma}\label{l.decomp}
Let $\{\nu^+_{E, x}\}_{x\in M}$ be a decomposition of the stationary measure $\nu_E^+$ into the measures on the fibers. Then for any compact $J\subseteq\R$, there exit $C>0, \alpha>0$ such that for any $E\in J$ and  $m$-a.e. $x\in M$ the measure $\nu^+_{E, x}$ on $\Sc$ is $(C, \alpha)$-H\"older continuous. The same holds for the decomposition $\{\nu^-_{E, x}\}_{x\in M}$ of the measure $\nu^-_E$.
\end{lemma}
\begin{proof} The proof is a straightforward application of the results from \cite{GKM}. Namely, denote by $\mM$ the space of all probability measures on $\Diff^1(\Sc)$ equipped with weak-$\ast$ topology. For a measure $\nu$ on $\Sc$ and a distribution $\mu\in \mM$ denote $\mu*\nu=\int_{\Diff^1(M)} (f_* \nu) \, d\mu(f).$ The following result holds for smooth random dynamical systems on arbitrary manifold, but here we formulate it specifically for the case of the circle:

 \begin{theorem}[Theorem 2.8 from \cite{GKM}]\label{t:main-3}
    Let $\mathcal{K}\subset \mM$ be a compact, satisfying the following conditions:
    \begin{itemize}
        \item \textbf{(uniform finite moment condition)} There exists $\gamma>0$, $C_0$ such that for any $\mgr\in\mathcal{K}$ one has
        \begin{equation} \label{PosMomentCondComp}
            \int \mathfrak{L}(f)^\gamma \, d \mgr(f) < C_0,
        \end{equation}
        where $\mathfrak{L}(f)=\max_{y\in S^1}\max\left(|f'(y)|,  |(f^{-1})'(y)|\right).$
        \vspace{3pt}
        \item \textbf{(no deterministic images)} For any $\mgr\in\mathcal{K}$ there are no probability measures $\msp,\msp'$ on~$\Sc$ such that $f_* \msp = \msp'$ for $\mgr$-almost all $f\in \Diff^1(\Sc)$.
    \end{itemize}
    Then there exist $\alpha>0$, $C$, $\kappa<1$ such that for any initial measure $\msp_0$, any $n$, and any distributions $\mgr_1,\dots, \mgr_n\in \mathcal{K}$ the $n$-th image of $\msp_0$ satisfies $(\alpha,C)$-H\"older property on the scales up to $\kappa^n$:
    \[
        \forall x\in M \quad \forall r>\kappa^n \quad (\mgr_n * \dots * \mgr_1 * \msp_0)(B_r(x)) < Cr^{\alpha}.
    \]
\end{theorem}
For any $a\in \R$, denote by $\mu_a^E$ the distribution on $\Diff^1(\Sc)$ formed by projectivizations of the linear maps $\begin{pmatrix}
  E-a -W_\omega & -1 \\
  1 & 0 \\
\end{pmatrix},$ where $W_\omega$ is a random variable distributed with respect to the measure $\xi$. Since $M$ is compact and $\varphi:M\to \R$ is continuous, $\varphi$ must be bounded, $|\varphi|\le B$. Collection of measures $\{\mu_a^E\}_{E\in J, |a|\le B}$ is compact and satisfies uniform finite moment condition (since $\int|x|^\gamma d\xi<\infty$) from Theorem \ref{t:main-3}. Moreover, the collection of convolutions $\{\mu_a^E*\mu_b^E\}_{|a|, |b|\le B, E\in J}$ also satisfies no deterministic images condition (due to specific form of the transfer matrices, see \cite[Section 4.1]{GK3}).

For $m$-almost every $x\in M$ we have
$$
\nu^+_{E,x}=\lim_{n\to \infty}\mu^E_{\varphi(G^{-1}(x))}*\ldots * \mu^E_{\varphi(G^{-n+1}(x))}*\nu^+_{E, G^{-n}(x)}.
$$
Therefore,  Theorem \ref{t:main-3} implies that for $m$-almost every $x\in M$, the measure $\nu^+_{E, x}$ is $(C, \alpha)$-H\"older on any scale, hence $(C, \alpha)$-H\"older continuous.

By passing to the dynamics by inverse maps and applying exactly the same arguments, we get the statement for  the decomposition $\{\nu^-_{E, x}\}_{x\in M}$ of the measure~$\nu^-_E$.
\end{proof}

We have the following formula for an increment of the rotation number:
\begin{equation}\label{e.last}
\rho(E_2)-\rho(E_1)=\E \int_{M\times \Sc} \Phi_{\nu^-_{E_2, x}}(\tg_{E_2, x, \omega}(\ty), \tg_{E_1, x, \omega}(\ty))d\nu^+_{E_1}(x,y),
\end{equation}
which can be shown using almost verbatim the arguments used in the proof of Theorem \ref{t.main2}.

Similarly to the proof of Theorem \ref{c:iid-Holder}, Lipschiz continuity of functions $g_{E,  x, \omega}$ in the parameter $E\in J$, H\"older continuity of the fibered stationary measures $\{\nu^+_{E, x}\}_{x\in M}$ and $\{\nu^-_{E, x}\}_{x\in M}$ guaranteed by Lemma \ref{l.decomp}, and the formula (\ref{e.last}) for the rotation number  imply H\"older regularity of the rotation number, and hence the local H\"older regularity of the IDS of the family of operators $\{H_x+W_\omega\}_{x\in X, \omega\in \Omega}$.


\end{proof}


\section*{Acknowledgments}
We would like to thank the Institut Mittag-Leffler for hospitality; the first two authors participated in the program  ``Two dimensional maps''  that was organized there in Spring 2023, and that is when this project was initiated. Also, we are grateful to Lana Jitomirskaya for providing several useful references, to Dominique Malicet for his remark that lead to Remark \ref{r.DM}, and to Jairo Bochi, Serge Cantat, \'Etienne Ghys, and Jake Fillman for their attention to our work.

\end{document}